\documentclass[10pt]{article}

\usepackage{amsfonts,amsmath,amssymb,mathrsfs,dsfont,amsthm}
\usepackage{graphicx,ifthen}
\usepackage{color}
\usepackage{array}
\usepackage{natbib}
\usepackage{enumerate}
\usepackage{stmaryrd}
\usepackage{bbm,bm}
\usepackage{comment}
\usepackage{caption}
\usepackage{booktabs}

\usepackage{xspace}
\RequirePackage[colorlinks,citecolor=blue,urlcolor=blue]{hyperref}
\usepackage{booktabs}

\newtheorem{theorem}{Theorem}
\theoremstyle{definition}

\newcommand{\E}{\mathrm{E}}
\newcommand{\R}{\mathbb{R}}
\newcommand{\N}{\mathbb{N}}

\newcommand{\law}{\mathcal{L}\xspace}

\newcommand{\Var}{\mathrm{Var}}

\newcommand{\p}{p}
\newcommand{\Small}[1]{\textstyle #1 \displaystyle}
\newcommand{\ddr}{\mathrm{d}}
\newcommand{\edr}{\mathrm{e}}
\newcommand{\fine}{\hfill $\Box$}

\newcommand{\comillas}[1]{``\,#1\,''}
\definecolor{iblue}{rgb}{0.1,0,0.75}
\definecolor{ired}{rgb}{0.9,0,0.1}

\newcommand{\Prob}{\mathrm{P}}

\def\nsample{n}
\def\latent{X_{1:\nsample}\xspace}
\def\data{Y_{1:\nsample}\xspace}
\def\latentunique{X_{1:k}^*\xspace}
\def\frequencies{\nsample_{1:k}^*\xspace}

\usepackage[acronym]{glossaries}

\makenoidxglossaries 

\newacronym{iid}{i.i.d.}{independent and identically distributed}
\newacronym{CLT}{CLT}{central limit theorem}

\graphicspath{ {graphics/} }

\allowdisplaybreaks

\begin{document}

\begin{center}
{\large{\sc Stochastic approximations to the Pitman--Yor process}}
\bigskip

Julyan Arbel, Pierpaolo De Blasi, Igor Pr\"{u}nster 
\bigskip

\end{center}


\begin{abstract}
In this paper we consider approximations to the popular Pitman--Yor process obtained by truncating the stick-breaking representation. The truncation is determined by a random stopping rule 
that achieves an almost sure control on the approximation error in total variation distance. We derive the asymptotic distribution of the random truncation point as the approximation error $\epsilon$ goes to zero in terms of a polynomially tilted positive stable random variable. The practical usefulness and effectiveness of this theoretical result is demonstrated by devising a sampling algorithm to approximate functionals of the $\epsilon$-version of the Pitman--Yor process.
\end{abstract}

\\[-2mm]

\section{Introduction\label{sec:intro}}

The Pitman--Yor process defines a rich and flexible class of random probability measures used as prior distribution in Bayesian nonparametric inference. It originates from the work of \citet{Per:Pit:Yor:92}, further investigated in \citet{Pit:95,Pit:Yor:97}, and 
its use in nonparametric inference was initiated by \citet{Ish:Jam:01}. Thanks to its analytical tractability and flexibility, it has found applications in a variety of inferential problems which include species sampling \citep{Lij:etal:07,Fav:etal:09,Nav:etal:08}, survival analysis and gene networks \citep{Jar:etal:10, yang}, linguistics and image segmentation \citep{Teh:06,Sud:Jor:09}, curve estimation \citep{Can:etal:17} and time-series and econometrics \citep{caron,leisen}. 
The Pitman--Yor process is a discrete probability measure 
\begin{equation}\label{eq:discrete}
  P(\ddr x)=\sum_{i\geq1}p_{i}\delta_{\xi_{i}}(\ddr x)
\end{equation}
where $(\xi_{i})_{i\geq1}$ are \gls{iid} random variables with common distribution $P_0$ on a 
Polish space ${\cal X}$, and $(p_{i})_{i\geq1}$ are random frequencies, i.e. $p_i\geq 0$ and $\sum_{i\geq1}p_{i}=1$, independent of $(\xi_{i})_{i\geq1}$. The distribution of the frequencies of the Pitman--Yor process is known in the literature as the two-parameter Poisson--Dirichlet distribution.
Its distinctive property is that the frequencies in {\it size-biased order}, that is the random arrangement in the order of appearance in a simple random sampling without replacement, admit the \textit{stick-breaking representation}, or residual allocation model,
\begin{equation}\label{eq:stick-breaking}
  \p_i\overset{d}{=}V_i\prod_{j=1}^{i-1}(1-V_j),
  \quad V_j\overset{\mbox{\footnotesize ind}}{\sim}
  \mbox{beta}(1-\alpha,\theta+j\alpha) 
\end{equation}
for $0\leq \alpha<1$ and $\theta>-\alpha$, see \citet{Pit:Yor:97}. By setting $\alpha=0$ one recovers the Dirichlet process of \citet{Fer:73}.
Representation \eqref{eq:stick-breaking} turns out very useful in devising finite support approximation to the Pitman--Yor process obtained by  truncating the summation in~\eqref{eq:discrete}. A general method consists in setting the truncation level $n$ by replacing $p_{n+1}$ with $1-(p_1+\cdots+p_n)$ in~\eqref{eq:discrete}. The key quantity is the \textit{truncation error} of the infinite summation~\eqref{eq:discrete},
\begin{equation}\label{eq:R_n}
  R_n= \sum_{i>n} \p_i  
  = \prod_{j\leq n} (1-V_j),
\end{equation}
since the resulting truncated process, say $P_n(\cdot)$, will be close to $P(\cdot)$ according to
  $|P(A)-P_n(A)|\leq R_n$
for any measurable $A\subset{\cal X}$. It is then important to study the distribution of the truncation error $R_n$ as $n$ gets large in order to control the approximation error. \citet{Ish:Jam:01} proposes to determine the truncation level based on the moments of $R_n$. Cf. also \citet{Ish:Zar:02,Gel:Kot:02}. In this paper we propose and investigate a random truncation by setting $n$ such that $R_n$ is smaller than a predetermined value $\epsilon\in(0,1)$ with probability one. Specifically, we define 
\begin{equation}\label{eq:tau_eps}
  \tau(\epsilon)=\min\{n\geq 1:\
  R_n<\epsilon\}
\end{equation}
as the stopping time of the multiplicative process $(R_n)_{n\geq 1}$ and, 
following  \citet[Section 4.3.3]{Gho:vdV:17}, we call {\it $\epsilon$-Pitman--Yor ($\epsilon$-PY) process} the Pitman--Yor process truncated at $n=\tau(\epsilon)$, namely
\begin{equation}\label{eq:eps-PY}
  P_\epsilon(\ddr x)
  =\sum_{i=1}^{\tau(\epsilon)}\p_i
  \delta_{\xi_i}(\ddr x)+ R_{\tau(\epsilon)}
  \delta_{\xi_0}(\ddr x),
\end{equation}  
where $\xi_0$ has distribution $P_0$, independent of the sequences $(\p_i)_{i\geq 1}$ and $(\xi_i)_{i\geq 1}$. 
By construction, $P_\epsilon$ is the finite stick-breaking approximation to $P$ with the smallest number of support points given a predetermined approximation level. In fact $\tau(\epsilon)$ controls the error of approximation according to the total variation bound
\begin{equation}\label{eq:TV}
  d_{TV}(P_\epsilon,P)
  =\sup_{A\subset{\cal X}}|P(A)-P_\epsilon(A)|
  \leq \epsilon
\end{equation}
almost surely (a.s.). As such, it also guarantees the almost sure convergence of measurable functionals of $P$ by the corresponding functionals of $P_\epsilon$ as $\epsilon\to 0$, cf. \citet[Proposition 4.20]{Gho:vdV:17}. A typical application is in Bayesian nonparametric inference on mixture models where the Pitman--Yor process is used as prior distribution on the mixing measure. The approximation $P_\epsilon$ can be applied to the posterior distribution given the latent variables, cf. Section \ref{sec:2.2} for details. 
In the Dirichlet process case, $P_\epsilon$ has been studied by \citet{Mul:Tar:98}. In this setting $\tau(\epsilon)-1$ is Poisson distributed with parameter $\theta\log 1/\epsilon$, which makes an exact sampling of the $\epsilon$-approximation \eqref{eq:eps-PY} feasible. This has been implemented in the highly popular R software DPpackage, see \citep{Jar:07,Jar:etal:11}, to draw posterior inference on the random effect distribution of linear and generalized linear mixed effect model.
Finally, in \citet{All:Zar:14} a  different type of finite dimensional truncation of the Pitman--Yor process based on decreasing frequencies has been proposed, see Section \ref{sec:5} for a discussion.

The main theoretical contribution of this paper is the derivation of the asymptotic distribution of $\tau(\epsilon)$ as $\epsilon\to 0$ for $\alpha>0$. 
As \eqref{eq:tau_eps} suggests, the asymptotic distribution of $\tau(\epsilon)$ is related to that of $R_n$ in \eqref{eq:R_n} as $n\to\infty$. According to \citet[Lemma 3.11]{Pit:06}, the latter involves a polynomially tilted stable random variable $T_{\alpha,\theta}$, see Section \ref{sec:2} for a formal definition. The main idea is to work with $T_n=-\log R_n$ so to deal with sums of the independent random variables $Y_i=-\log(1-V_i)$. The distribution of  $\tau(\epsilon)$ can be then studied in terms of the allied {\it renewal counting process}
  $N(t)=\max\{n:\ T_n\leq t \}$,
according to the relation
  $\tau(\epsilon)=N(\log 1/\epsilon)+1$.
The problem boils down to the derivation of an appropriate a.s. convergence of $N(t)$ as $t\to\infty$, which, in turn, is obtained from the asymptotic distribution of $T_n$ by showing that $N(t)\to\infty$ a.s. as $t\to\infty$ together with a (non standard) application of the law of large numbers for randomly indexed sequences. This strategy proves successful in establishing the almost sure convergence of $\tau(\epsilon)-1$ to 
  $(\epsilon T_{\alpha,\theta}/\alpha)^{-\alpha/(1-\alpha)}$
as $\epsilon\to 0$.
The form of the asymptotic distribution reveals how large the truncation point $\tau(\epsilon)$ is as $\epsilon$ gets small in terms of the model parameters $\alpha$ and $\theta$.
In particular, it highlights the power law behavior of $\tau(\epsilon)$ as $\epsilon\to 0$, namely the growth at the polynomial rate $1/\epsilon^{\alpha/(1-\alpha)}$ compared to the slower logarithmic rate $\theta\log 1/\epsilon$ in the Dirichlet process case. This is further illustrated by a simulation study in which we generate from the asymptotic distribution of $\tau(\epsilon)$ by using Zolotarev's integral representation of the positive stable distribution as in \citet{Dev:09}. 
As far as the simulation of the $\epsilon$-PY process is concerned, exact sampling is feasible by implementing the stopping rule in \eqref{eq:tau_eps}, that is by simulating the stick breaking frequencies $p_j$ until the error $R_n$ crosses the approximation level $\epsilon$. As this can be computationally expensive when $\epsilon$ is small, as an alternative we propose to use the asymptotic distribution of $\tau(\epsilon)$ by simulating the truncation point first, then run the stick breaking procedure up to that point. It results in an approximate sampler of the $\epsilon$-PY process that we compare with the exact sampler in a simulation study involving moments and mean functionals. 

The rest of the paper is organized as follows. In Section~\ref{sec:2}, we derive the asymptotic distribution of $\tau(\epsilon)$ and explain how to use it to simulate from the $\epsilon$-PY process. Section~\ref{sec:3} reports a simulation study on the  distribution of $\tau(\epsilon)$ and on functionals of the $\epsilon$-PY process. In Section \ref{sec:4}, to help the understanding and gain additional insight on the asymptotic distribution, we highlight the connections of $\tau(\epsilon)$ with Pitman's theory on random partition structures.
We conclude with a discussion of open problems in Section \ref{sec:5}. 
The details of Devroye's algorithm for generating from a polynomially tilted positive stable random variable are given in Appendix \ref{appendix:devroye}.


\section{Theory and algorithms}\label{sec:2}


\subsection{Asymptotic distribution of $\tau(\epsilon)$}

In this section we derive the asymptotic distribution of the stopping time $\tau(\epsilon)$ and show how to simulate from it. 
We start by introducing the renewal process interpretation which is crucial for the asymptotic results. As explained in the previous section, in order to study the distribution of $\tau(\epsilon)$ it is convenient to work with the log transformation of the truncation error $R_n$ in \eqref{eq:R_n}, that is
\begin{equation}\label{eq:T_n}
  T_n=\sum_{i=1}^n Y_i,\quad
  Y_i=-\log(1-V_i),
\end{equation}  
with
  $V_j\overset{\mbox{\footnotesize ind}}{\sim}\mbox{beta}(1-\alpha,\theta+j\alpha)$
as in \eqref{eq:stick-breaking}. Being a sum of  independent and nonnegative random variables, $(T_n)_{n\geq 1}$ takes the interpretation of a (generalized) renewal process with independent waiting times $Y_i$. For $t\geq 0$ define 
\begin{equation}\label{eq:N(t)}
  N(t)=\max\{n:\ T_n\leq t \},
\end{equation}  
to be the {\it renewal counting process} associated to $(T_n)_{n\geq 1}$, which is related to $\tau(\epsilon)$ via
  $\tau(\epsilon)=N(\log 1/\epsilon)+1$.
Classical renewal theory pertains to iid waiting times while here there is no identity in distribution unless $\alpha=0$, i.e. the Dirichlet process case.  
In the latter setting, one gets $Y_i\overset{\mbox{\footnotesize \gls{iid}}}{\sim}\mbox{Exp}(\theta)$ so that $T_n$ has gamma distribution with scale parameter $n$. We immediately get from the relation 
  $\{T_n\leq t\}=\{N(t)\geq n\}$
that $N(t)\sim \mbox{Pois}(\theta t)$ and, in turn, that $\tau(\epsilon)-1$ has $\mbox{Pois}(\theta\log(1/\epsilon))$ distribution. As far as asymptotics is concerned, $T_n$ satisfies the \gls{CLT} with
  $(T_n-n/\theta)/(\sqrt{n}/\theta)
  \to_d Z$
where $Z\sim\mbox{N}(0,1)$. The asymptotic distribution of $N(t)$ can be obtained via Ascombe theorem, cf. \citet[Theorem 7.4.1]{Gut:13}, to get
  $(N(t)-\theta t)/(\sqrt{\theta t})
  \rightarrow_d Z$,
as $t\to\infty$, in accordance with the standard normal approximation of the Poisson distribution with large rate parameter.

In the general Pitman--Yor case $\alpha>0$, the waiting times $Y_i$ are no more identically distributed.
More importantly, generalizations of the CLT such as the Lindeberg--Feller theorem 
do not apply for $T_n$, 
hence we cannot resort to Anscombe's theorem to derive the asymptotic distribution of $N(t)$ and, in turn, of $\tau(\epsilon)$. 
Nevertheless, the limit exists but is not normal as stated in Theorem \ref{th:1} below. 
To this aim, let $T_\alpha$ be a positive 
stable random variable with exponent $\alpha$, that is $\E (\edr^{-s T_\alpha})=\edr^{-s^\alpha}$, and denote its density by $f_\alpha(t)$. A polynomially tilted version of $T_\alpha$ is defined as the random variable $T_{\alpha,\theta}$ with density proportional to $t^{-\theta} f_\alpha(t)$, that is
\begin{equation}\label{eq:density}
  f_{\alpha,\theta}(t) 
  = \frac{\Gamma(\theta+1)}{\Gamma(\theta/\alpha+1)}
  t^{-\theta}f_\alpha(t),\quad t> 0.
\end{equation}
The random variable $T_{\alpha,\theta}$ is of paramount importance in the theory of random partition structures associated to the frequency distribution of the Pitman--Yor process, see Section \ref{sec:4} for details. In particular, the convergence of $R_n$ can be expressed in terms of $T_{\alpha,\theta}$. In Theorem \ref{th:1}
the a.s. limit of $\log N(t)$ as $t\to\infty$ is obtained from that of $T_n=-\log R_n$ as $n\to\infty$ by showing that $N(t)\to\infty$ a.s. as $t\to\infty$ and by an application of the law of large numbers for randomly indexed sequences.
\begin{theorem}\label{th:1}
Let $N(t)$ be defined in \eqref{eq:T_n}--\eqref{eq:N(t)} and $T_{\alpha,\theta}$ be the random variable with density in \eqref{eq:density}. Then
  $t-(1/\alpha-1)\log N(t)+\log\alpha
  \to_{a.s.} \log T_{\alpha,\theta}$
as $t\to\infty$.
\end{theorem}
\begin{proof}
By definition \eqref{eq:N(t)}, the renewal process $N(t)$ is related to the sequence of renewal epochs $T_n$ through 
\begin{equation}\label{eq:fundamental}
  \{T_n\leq t\}=\{N(t)\geq n\}.
\end{equation}
Since $N(T_n)=n$, we have $T_{N(t)}=T_n$ when $t=T_n$, thus 
  $0=t-T_{N(t)}$
for $t=T_n$.
Moreover, since $N(t)$ is increasing, when $T_n<t<T_{n+1}$, $N(T_n)<N(t)<N(t)+1$, hence
  $T_{N(t)}<t<T_{N(t)+1}$,
i.e.
  $0<t-T_{N(t)}<T_{N(t)+1}-T_{N(t)}=Y_{N(t)+1}$.
Together the two relations above yield
\begin{equation}\label{eq:sandwich}
  0\leq t-T_{N(t)}<Y_{N(t)+1}.
\end{equation}  
From Lemma 3.11 of \citet{Pit:06} and an application of the continuous mapping theorem \citep[see Theorem 10.1 in][]{Gut:13} the asymptotic distribution of $T_n$ is obtained as
  $$ T_n-(1/\alpha-1)\log n+\log\alpha
  \to_{a.s.} \log T_{\alpha,\theta}
  \quad\mbox{as }n\to\infty.$$
Now we would like to take the limit with respect to $n=N(t)$ as $t\to \infty$, that is apply the law of large numbers for randomly indexed sequence \citep[see Theorem 6.8.1 in][]{Gut:13}. To this aim, we first need to prove that 
  $N(t)\to_{a.s.}\infty$
as $t\to\infty$.
Since $N(t)$ is non decreasing, by an application of Theorem 5.3.5 in \citet{Gut:13}, it is sufficient to prove that $N(t)\to\infty$ in probability as $t\to\infty$, that is
  $\Prob(N(t)\geq n)\to 1$
as $t\to\infty$
for any $n\in\N$. But this is an immediate consequence of  the inversion formula \eqref{eq:fundamental}. We have then established that 
  $$T_{N(t)}-(1/\alpha-1)\log N(t)+\log\alpha
  \to_{a.s.} \log T_{\alpha,\theta}\quad
  \mbox{as }t\to\infty.$$
To conclude the proof, we need to replace $T_{N(t)}$ with $t$ in the limit above. Note that, from \eqref{eq:sandwich},
  $|t-T_{N(t)}|\leq Y_{N(t)+1}$
so it is sufficient to show that the upper bound goes to zero a.s.. Actually, by a second application of Theorem 6.8.1 in \citet{Gut:13} it is sufficient to show that 
  $Y_n\to_{a.s.} 0$
as $n\to\infty$.
This last result is established as follows. Recall that 
  $Y_j=-\log(1-V_j)$ for
  $V_j\overset{\mbox{\footnotesize ind}}{\sim}\mbox{beta}(1-\alpha,\theta+j\alpha)$.
For $\epsilon>0$,
\begin{align}
  \nonumber
  \Prob(1-V_n<\edr^{-\epsilon})
  &=\int_0^{\edr^{-\epsilon}}
  \frac{\Gamma(\theta+n\alpha+1-\alpha)}
  {\Gamma(\theta+n\alpha)\Gamma(1-\alpha)}
  v^{\theta+n\alpha-1}(1-v)^{-\alpha}\ddr x\\
  \nonumber
  &\leq \frac{(1-\edr^{-\epsilon})^{-\alpha}}
  {\Gamma(1-\alpha)}
  \frac{\Gamma(\theta+n\alpha+1-\alpha)}
  {\Gamma(\theta+n\alpha)}
  \frac{\edr^{-\epsilon(\theta+n\alpha)}}{\theta+n\alpha}\\
  \label{eq:BC}
  &=\frac{(1-\edr^{-\epsilon})^{-\alpha}}{\Gamma(1-\alpha)}
  (\theta+n\alpha)^{-\alpha}
  \edr^{-\epsilon(\theta+n\alpha)}
  \Big(1+O\Big(\frac{1}{\theta+n\alpha}\Big)\Big)
\end{align}
where 
in equality \eqref{eq:BC}
we have used Euler's formula
  $$\Gamma(z+\alpha)/\Gamma(z+\beta)=z^{\alpha-\beta}\left[
  1+\frac{(\alpha-\beta)(\alpha+\beta-1)}{2z}+O(z^{-2})\right]$$
for $z\to\infty$, see \citet{Tri:Erd:51}. 
Since 
  $\Prob(Y_n>\epsilon)=\Prob(1-V_n<\edr^{-\epsilon})$,
\eqref{eq:BC} implies that $\Prob(Y_n>\epsilon)$ is exponentially decreasing in $n$ and, in turn, that $\sum_{n\geq 1}\Prob(Y_n>\epsilon)<\infty$. An application of Borel--Cantelli Lemma yields $Y_n\to_{a.s.} 0$ and the proof is complete.
\end{proof} 

The asymptotic distribution of $\tau(\epsilon)$ is readily derived from Theorem \ref{th:1} via the formula
  $\tau(\epsilon)=N(\log 1/\epsilon)+1$
and an application of the continuous mapping theorem. The proof is omitted.
\begin{theorem}\label{cor:limiting_tau_eps}
Let $\tau(\epsilon)$ be defined in \eqref{eq:tau_eps} and $T_{\alpha,\theta}$ be the random variable with density in \eqref{eq:density}. Then
  $\tau(\epsilon)-1
  \sim_{a.s.}
  (\epsilon T_{\alpha,\theta}/\alpha)^{-\alpha/(1-\alpha)}$
as $\epsilon\to 0$.
\end{theorem}
In order to sample from the asymptotic distribution of $\tau(\epsilon)$, the key ingredient is random generation from the polynomially tilted stable random variable $T_{\alpha,\theta}$. 
Following \citet{Dev:09}, we resort to Zolotarev's integral representation, so let $A(u)$ be the Zolotarev function
  $$A(x)=\bigg(
  \frac{\sin(\alpha x)^\alpha\sin((1-\alpha)x)^{1-\alpha}}{\sin(x)}
  \bigg)^{\frac{1}{1-\alpha}},\quad x\in[0,\pi]$$
and $Z_{\alpha,b}$, $\alpha\in(0,1)$ and $b>-1$ be a Zolotarev random variable with density given by
  $$f(x)=\frac{\Gamma(1+b\alpha)\Gamma(1+b(1-\alpha))}
  {\pi\Gamma(1+b)A(x)^{b(1-\alpha)}},\quad 
  x\in[0,\pi].$$
According to Theorem 1 of \citet{Dev:09}, for $G_a$ a gamma distributed  random variable with shape $a>0$ and unit rate,
  $$T_{\alpha,\theta}\overset{d}{=}
  \bigg( \frac{A(Z_{\alpha,\theta/\alpha})}
  {G_{1+\theta(1-\alpha)/\alpha}}
  \bigg)^{\frac{1-\alpha}{\alpha}}$$
so that random variate generation simply requires one gamma random variable and one Zolotarev random variable. For the latter, rejection sampler can be used as detailed in \citet{Dev:09}. See {\sc Algorithm 3} in Appendix \ref{appendix:devroye}. 


\subsection{Simulation of the $\epsilon$-PY process}
\label{sec:2.2}

Given $\alpha,\theta,\epsilon$ and a probability measure $P_0$ on ${\cal X}$, an $\epsilon$-PY process can be generated by implementing the stopping rule in the definition of $\tau(\epsilon)$, cf. \eqref{eq:tau_eps}. The algorithm consists in a while loop as follows:

\begin{quote}
\centerline{{\sc Algorithm 1} (Exact sampler of $\epsilon$-PY)}
\vspace{4pt}

{\tt
1. set $i= 1$, $R= 1$
\vspace{4pt}

2. while $R\geq\epsilon$:$\quad$ generate $V$ 
  from $\text{\rm beta}(1-\alpha,\theta+i\alpha)$.
\\
\phantom{3. while $R>\epsilon$:$\quad$} set 
  $p_i= VR$, $R= R(1-V)$, $i= i+1$
\vspace{4pt}

3. set 
  $\tau= i$, $R_\tau=R$
\vspace{4pt}

4. generate $\tau+1$ random variates $\xi_{0},\xi_{1},\ldots,\xi_{\tau}$ from $P_0$
\vspace{4pt}

5. set
  $  P_\epsilon(\ddr x)
  =\sum_{i=1}^{\tau}\p_i
  \delta_{\xi_i}(\ddr x)+ R_\tau
  \delta_{\xi_0}(\ddr x)$
}
\end{quote}

When $\epsilon$ is small, the while loop happens to be computationally expensive since conditional evaluations at each iteration slow down computation, and memory allocation for the frequency and location vectors cannot be decided beforehand.
In order to avoid these pitfalls and make the algorithm faster, one should generate the stopping time $\tau(\epsilon)$ first, and the frequencies up to that point later. We propose to exploit the asymptotic distribution of $\tau(\epsilon)$ in Theorem~\ref{cor:limiting_tau_eps} 
as follows:

\begin{quote}
\centerline{{\sc Algorithm 2} (Approximate sampler of $\epsilon$-PY)}
\vspace{4pt}

{\tt
1: generate $T\overset{d}{=} T_{\alpha,\theta}$
\vspace{4pt}

2: set 
  $\tau\leftarrow 1+\lfloor 
  (\epsilon T/\alpha)^{-\alpha/(1-\alpha)}
  \rfloor$
\vspace{4pt}

3. for $i=1,\ldots,\tau$:$\quad$ generate $V_i$ 
  from $\text{\rm beta}(1-\alpha,\theta+i\alpha)$.
\\
\phantom{3. for $i=1,\ldots,\tau$:$\quad$} set 
  $p_i= V_i\prod_{j=1}^{i-1}(1-V_j)$
\vspace{4pt}

4. set 
  $R_\tau= 1-\sum_{i=1}^\tau p_i=\prod_{i=1}^\tau (1-V_j)$
\vspace{4pt}

5: generate $\tau+1$ random variates $\xi_{0},\xi_{1},\ldots,\xi_{\tau}$ from $P_0$
\vspace{4pt}

6: set
  $  P_\epsilon(\ddr x)
  =\sum_{i=1}^{\tau}\p_i
  \delta_{\xi_i}(\ddr x)+ R
  \delta_{\xi_0}(\ddr x)$
}
\end{quote}

{\sc Algorithm 2} is an {\it approximate} sampler of the $\epsilon$-PY process (while {\sc Algorithm 1} is an {\it exact} one) since it introduces two sources of approximations. First, through the use of the asymptotic distribution of $\tau(\epsilon)$. Second, through {\tt Step 3} since the $V_i$'s are not generated according to the conditional distribution given $\tau(\epsilon)$,
rather unconditionally. Finding the conditional distribution of $V_i$, or an asymptotic approximation thereof, is not an easy task and is object of current research. In terms of the renewal process interpretation in \eqref{eq:T_n}--\eqref{eq:N(t)}, the problem is to generate the waiting times $Y_i=-\log(1-V_i)$, $i=1,\ldots,n$, from the conditional distribution of the renewal epochs $(T_1,\ldots,T_n)$ given $N(t)=n$ for $t=-\log 1/\epsilon$.

A typical use of samples from the Pitman--Yor process we have in mind is in infinite mixture models. In fact, the discrete nature of the Pitman--Yor process makes it a suitable prior on the \textit{mixing distribution}. {\sc Algorithm} 1 or {\sc Algorithm} 2 can be then applied to approximate a functional of the posterior distribution of the mixing distribution. In such models, the process components can be seen as latent features exhibited by the data.
Let $P$ denote such a process, $\nsample$ denote the sample size and  $X_{1:\nsample}=(X_1,\ldots,X_\nsample)$ be an exchangeable sequence from $P$, that is $X_{1:\nsample} | P \overset{\mbox{\footnotesize \gls{iid}}}{\sim} P$. Variables $X_{1:\nsample}$ are latent variables in a model conditionally on which observed data $Y_{1:\nsample}$ come from: $Y_j | X_j \overset{\mbox{\footnotesize ind}}{\sim} f(\,\cdot\, | X_j)$ 
where $f$ denotes a kernel density.
Actually, independence is not necessary here and applications also encompass dependent models such as Markov chain transition density estimation. 
In order to deal with the infinite dimensionality of the process, a strategy is to marginalize it and to draw posterior inference with a \textit{marginal} sampler. 
Since draws from a marginal sampler allows to make inference only on posterior expectations of the process, for more general functionals of $P$, in the form of $\psi(P)$, one typically needs to resort to an additional sampling step. Exploiting the composition rule
  $\law(\psi(P)|\data) 
  = \law(\psi(P)|\latent)\times\law(\latent|\data)$
this additional step boils down to sampling $P$ conditional on latent variables $\latent$. At this stage, recalling the conditional conjugacy of the Pitman--Yor process is useful. Among $\latent$, there are a number $k\leq \nsample$ of unique values that we denote by $\latentunique$. Let $\frequencies$ denote their frequencies. Then the following identity in distribution holds
  $$P|\latent = \sum_{j=1}^{k}q_j \delta_{X_j^*} + q_{k+1}P^*,$$
where, independently, $(q_1,\ldots,q_k,q_{k+1})\sim\mbox{Dirichlet}(\nsample_1^*-\alpha,\ldots,\nsample_k^*-\alpha,\theta+\alpha k)$ and $P^*$ is a Pitman--Yor process of parameter $(\alpha,\theta+\alpha k)$, see Corollary 20 of \citet{Pit:96a}. Thus sampling from $\law(P|\latent)$, hence from $\law(\psi(P)|\latent)$, requires sampling the infinite dimensional $P^*$.  Cf. \citet[Section 4.4]{Ish:Jam:01}. For the sake of comparison, the conjugacy of the Dirichlet process similarly leads to the need of sampling an infinite dimensional process, where $P|X_{1:\nsample}$ takes the form of a Dirichlet process. As already noticed, the truncation of the Dirichlet process is very well understood, both theoretically and practically. The popular \textsf{R} package \textsf{DPpackage} \citep{Jar:07,Jar:etal:11} makes use of the \textit{posterior} truncation point $\tau^*(\epsilon)$, as defined in \eqref{eq:eps-PY}, but here with respect to the posterior distribution of the process. Thus, it satisfies
  $\tau^*(\epsilon)-1\sim \mbox{Pois}((\theta+\nsample)
  \log(1/\epsilon)),$
where $\theta +\nsample$ is the precision of the posterior Dirichlet process. 
Adopting here similar lines for the Pitman--Yor process, we replace $P^*$ by the truncated process $P^*_\epsilon$
\begin{equation*}
  P^*_\epsilon(\ddr x)
  =\sum_{i=1}^{\tau^*(\epsilon)}\p^*_i
  \delta_{\xi_i}(\ddr x)+ R_{\tau^*(\epsilon)}
  \delta_{\xi_0}(\ddr x),
\end{equation*}  
cf. equation \eqref{eq:eps-PY}. Here $(\p_i^*)_{i\geq 1}$ are defined according to \eqref{eq:stick-breaking} with  $\theta+\alpha k$ in place of $\theta$, i.e. 
  $V_j\overset{\mbox{\footnotesize ind}}{\sim}
  \text{beta}(1-\alpha,\theta+\alpha (k+j))$.  
Hence, according to Theorem~\ref{cor:limiting_tau_eps} we have
\begin{equation*}
  \tau^*(\epsilon)-1\sim_{a.s.}
  (\epsilon T_{\alpha,\theta+\alpha k}/\alpha)^{-\alpha/(1-\alpha)},
  \quad\mbox{as }\epsilon\to 0
\end{equation*}
hence {\sc Algorithm 2} can be applied here.


\section{Simulation study}\label{sec:3}


\subsection{Stopping time $\tau(\epsilon)$}\label{sec:3.1}

According to Theorem \ref{cor:limiting_tau_eps}, the asymptotic distribution of $\tau(\epsilon)$ changes with $\epsilon$,  $\alpha$ and $\theta$. 
For illustration, we simulate $\tau(\epsilon)$ from {\tt Steps 1.-2.} in {\sc Algorithm 2} using Devroye's sampler, cf. {\sc Algorithm 3} in Appendix \ref{appendix:devroye}. In Figure \ref{fig:asymp-approx} we compare density plots obtained with $10^4$ iterations with respect to different combinations of $\epsilon$, $\alpha$ and $\theta$.
The plot in the left panel shows how smaller values of $\epsilon$ result in larger values of $\tau(\epsilon)$. In fact, as $\epsilon\to 0$, $\tau(\epsilon)$ increases proportional to $1/\epsilon^{\alpha/(1-\alpha)}$. Note also that $(\epsilon T_{\alpha,\theta}/\alpha)^{-\alpha/(1-\alpha)}$ is nonnegative for $T_{\alpha,\epsilon}<\alpha/\epsilon$, which happens with high probability when $\epsilon$ is small. As for $\alpha$, the plot in the central panel shows how $\tau(\epsilon)$ increases as $\alpha$ gets large. In fact, it is easy to see that $(\epsilon T_{\alpha,\theta}/\alpha)^{-\alpha/(1-\alpha)}$ is increasing in $\alpha$ when $T_{\alpha,\epsilon}<\edr^{1-\alpha}\alpha/\epsilon$, which also happens with high probability when $\epsilon$ is small, so the larger $\alpha$, the more stick-breaking frequencies are needed in order to account for a prescribed approximation error $\epsilon$. Finally, the plot in the right panel shows that the larger $\theta$, the larger $\tau(\epsilon)$. In fact, by definition, the polynomial tilting makes $T_{\alpha,\theta}$ stochastically decreasing in $\theta$.
\begin{figure}[!ht]
\centering
\begin{tabular}{ccc}
\includegraphics[width = .33\textwidth]{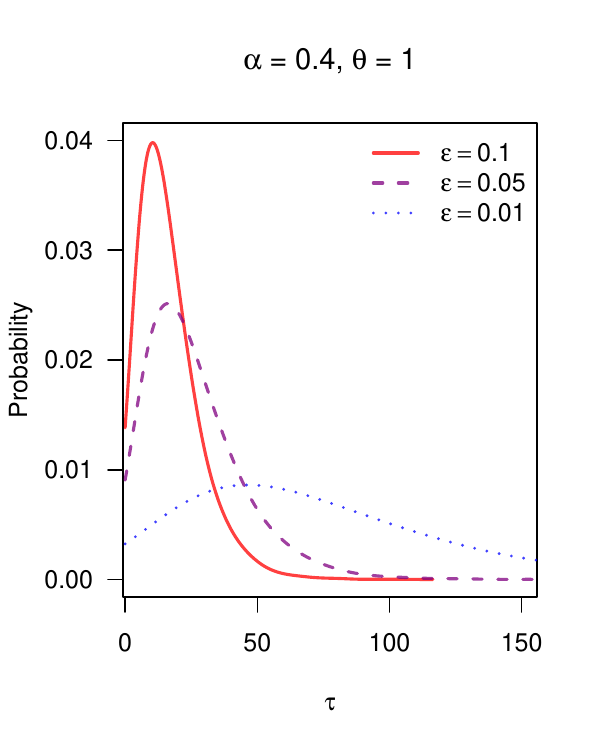}
\includegraphics[width = .33\textwidth]{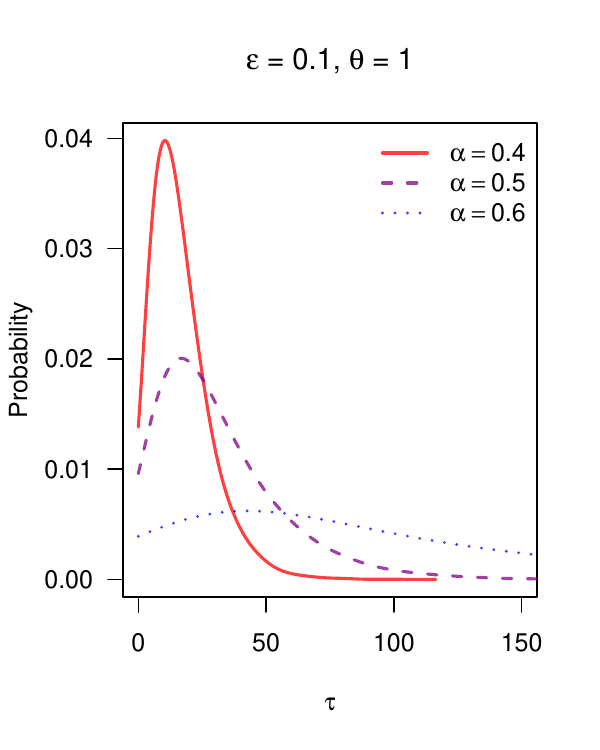}&
\includegraphics[width = .33\textwidth]{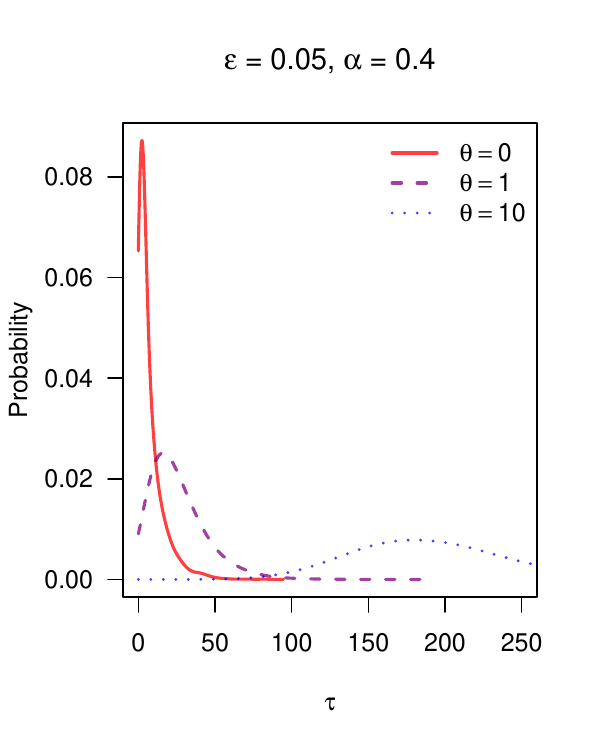}
\end{tabular}
\caption{
Density plot for the asymptotic approximation of $\tau(\epsilon)$ based on $10^4$ values under the following parameter configurations. 
Left: $\epsilon \in\{0.10, 0.05, 0.01\}$, $\alpha = 0.4$, $\theta = 1$.
Center: $\alpha \in\{0.4, 0.5, 0.6\}$, $\theta = 1$, $\epsilon = 0.1$.
Right: $\theta \in\{0,1,10\}$, $\alpha = 0.25$, $\epsilon = 0.05$.
}
\label{fig:asymp-approx}
\end{figure}

In order to illustrate the rate of convergence in Theorem \ref{cor:limiting_tau_eps}, we compare next the exact distribution of $\tau(\epsilon)$ with the asymptotic one. To do so, we repeat the following experiment several times: we simulate $\tau(\epsilon)$ from {\tt Steps 1.-3.} in {\sc Algorithm 1}, then we compare the empirical distribution of  
  $(\epsilon/\alpha)^{\alpha}
  (\tau(\epsilon)-1)^{1-\alpha}$
with
  $T_{\alpha,\theta}^{-\alpha}$,
the latter corresponding to the $\alpha$-diversity of the PY process, see Section \ref{sec:4} for a formal definition. In Table \ref{tab:diversity} we report the Kolmogorov distance together with expected value, median, first and third quartiles for both the exact and the asymptotic distribution obtained with $10^4$ iterations. This is repeated for $\alpha=0.5$, $\theta=\{0,1,10\}$ and $\epsilon=\{0.10,0.05,0.01\}$. As expected, as we decrease $\epsilon$, the Kolmogorov distance gets smaller to somehow different rates according to the parameter choice. The derivation of convergence rates is left for future research. 

\begin{table}[!ht]
\setlength{\tabcolsep}{3.5pt}
\vspace*{4pt}
\begin{center}
\begin{tabular}{cc@{\hskip 15pt}r@{\hskip 15pt}rr@{\hskip 15pt}rr@{\hskip 15pt}rr@{\hskip 15pt}rr}
\hline
  &&{\boldmath${d_K}$}& \multicolumn{2}{c}{\textbf{Mean}} & \multicolumn{2}{c}{\boldmath$25\%$} & 
\multicolumn{2}{c}{\textbf{Median}} & \multicolumn{2}{c}{\boldmath$75\%$} \\ 
 $\theta$ & $\epsilon$ &  & As & Ex & As & Ex & As & Ex & As & Ex \\
\hline
0 & 0.10 & 3.42 & 1.06 & 1.05 & 0.45 & 0.45 & 0.89 & 0.89 & 1.61 & 1.55 \\ 
  0 & 0.05 & 2.17 & 1.10 & 1.08 & 0.45 & 0.45 & 0.95 & 0.95 & 1.64 & 1.58 \\ 
  0 & 0.01 & 1.73 & 1.14 & 1.11 & 0.45 & 0.45 & 0.97 & 0.95 & 1.64 & 1.60 \\ 
  1 & 0.10 & 4.79 & 2.24 & 2.14 & 1.55 & 1.48 & 2.14 & 2.10 & 2.86 & 2.76 \\ 
  1 & 0.05 & 2.38 & 2.25 & 2.20 & 1.55 & 1.52 & 2.17 & 2.14 & 2.86 & 2.79 \\ 
  1 & 0.01 & 1.40 & 2.26 & 2.25 & 1.57 & 1.54 & 2.19 & 2.19 & 2.87 & 2.85 \\ 
  10 & 0.10 & 11.93 & 6.39 & 6.07 & 5.69 & 5.40 & 6.34 & 6.06 & 7.04 & 6.72 \\ 
  10 & 0.05 & 6.12 & 6.39 & 6.24 & 5.70 & 5.56 & 6.34 & 6.22 & 7.05 & 6.88 \\ 
  10 & 0.01 & 1.93 & 6.40 & 6.37 & 5.71 & 5.70 & 6.34 & 6.34 & 7.05 & 7.00 
\\
\hline
\end{tabular}
\end{center}
\caption{Summary statistics for the asymptotic distribution (As)  and exact distribution (Ex) of $\tau(\epsilon)$ at the scale of the $\alpha$-diversity based on $10^4$ values. The Kolmogorov distance ($d_K$) is between the empirical cumulative distribution function of the sample from the exact distribution and the asymptotic one (multiplied by a factor of 100). The parameter values are $\alpha = 0.5$, $\theta\in\{0,1,10\}$ and $\epsilon\in\{0.10,0.05,0.01\}$. 
}\label{tab:diversity}
\end{table}

\subsection{Functionals of the $\epsilon$-PY process}\label{sec:3.2}
In the case that $P$ is defined on $\mathcal{X}\subseteq\mathds{R}$, the total variation bound \eqref{eq:TV} implies that $|F(x)-F_\epsilon(x)|<\epsilon$ almost surely for any $x\in\mathds{R}$, where $F_\epsilon$ and $F$ are the cumulative distribution functions of $P_\epsilon$ and $P$. 
Also, measurable functionals $\psi(P)$ such as the mean $\mu=\int x P(\ddr x)$ can be approximated in distribution by the corresponding functionals $\psi(P_\epsilon)$. 
For illustration, we set $\mathcal{X}=[0,1]$ and $P_0$ the uniform distribution on $[0,1]$. For given $\alpha$ and $\theta$, we then compare the distribution under $P$ with that under the $\epsilon$-PY process $P_\epsilon$ for $F(1/2)$, $F(1/3)$ and $\mu=\int x P(\ddr x)$. 
As for the distribution of the finite dimensional distributions $F(1/2)$ and $F(1/3)$ under the full process $P$, we set $\alpha=0.5$ so to exploit  results in \citet{Jam:etal:10}. According to their Proposition 4.7,  the finite dimensional distributions of $P$ when $\alpha=0.5$ are given by
  $$f(w_1,\ldots,w_{n-1})=
  \frac{(\prod_{i=1}^np_i)\Gamma(\theta+n/2)}
  {\pi^{(n-1)/2}\Gamma(\theta+1/2)}
  \frac{w_1^{-3/2}\cdots w_{n-1}^{-3/2}
  (1-\sum_{i=1}^{n-1}w_i)^{-3/2}}
  {\mathcal{A}_n(w_1,\ldots,w_{n-1})^{\theta+n/2}}$$
for any partition $A_1,\ldots,A_n$ of $\mathcal{X}$ with $p_i=P_0(A_i)$ and 
  $\mathcal{A}_n(w_1,\ldots,w_{n-1})
  =\sum_{i=1}^{n-1}p_i^2 w_i^{-1}
  +p_n^2 (1-\sum_{i=1}^{n-1}w_i)^{-1}$.
Direct calculation shows that $F(1/2)$ has beta distribution with parameters $(\theta+1/2,\theta+1/2)$ while $F(1/3)$ has density
  $$f(w)=\frac{2}{\sqrt\pi}9^\theta
  \frac{\Gamma(\theta+1)}{\Gamma(\theta+1/2)}
  \frac{(w(1-w))^{\theta-1/2}}{(1+3w)^{\theta+1}}.$$
As for the mean functional $\mu=\int xP(\ddr x)$, the distribution under the full process $P$ is approximated by simulations by setting a deterministic truncation point sufficiently large. As for the distribution under $P_\epsilon$, we use both {\sc Algorithm 1} and {\sc Algorithm 2}. 

In Figure \ref{fig:random-probabilities} we compare the density plots of $F(1/2)$ for $\epsilon=\{0,1.0.05,0.001\}$ and $\theta=\{0,10\}$ under $P_\epsilon$ with the beta density under $P$ so to illustrate that the two distributions get close as $\epsilon$ gets small. 
\begin{figure}[!ht]
\centering
\begin{tabular}{c}
\includegraphics[width = \textwidth]{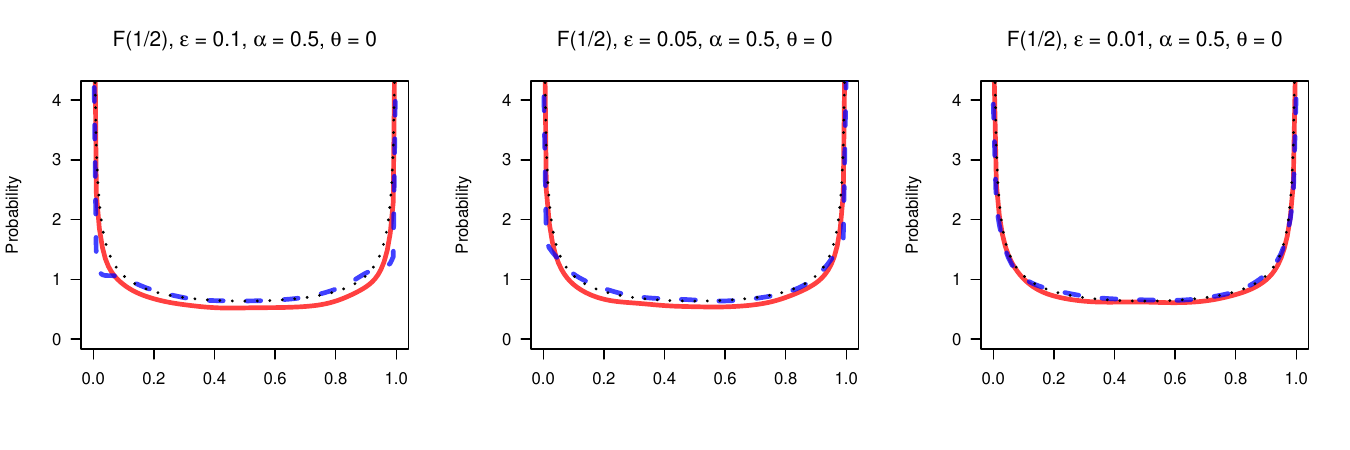}\\
\includegraphics[width = \textwidth]{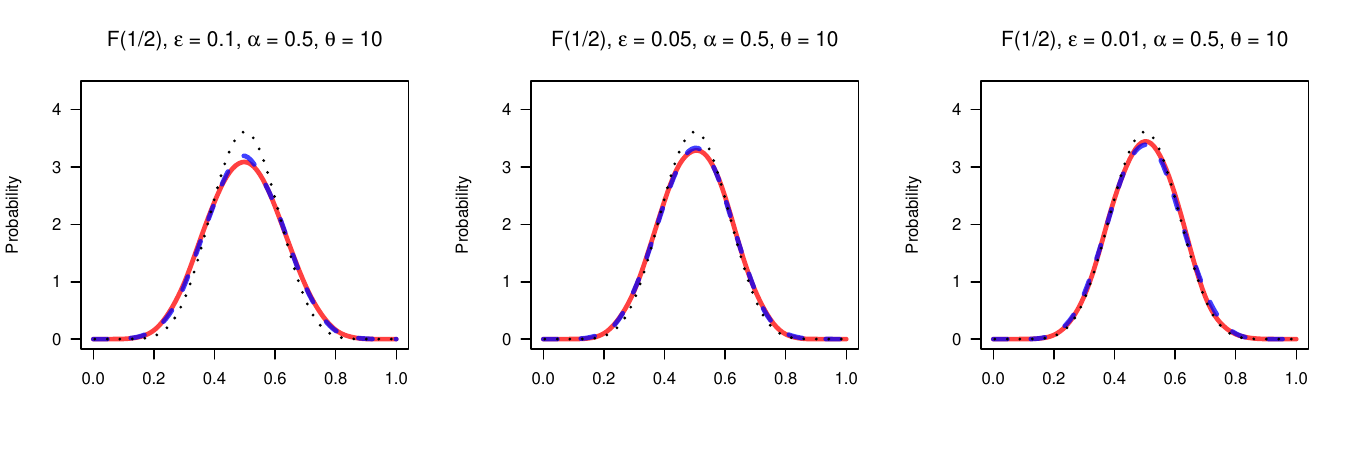}
\end{tabular}
\caption{Density plots for the random probability $F(1/2)$ using the {\sc Algorithm 2} (in red solid curve) and {\sc Algorithm 1} (in blue dashed curve) to sample from the $\epsilon$-PY process. The density under the Pitman--Yor process is the black dotted curve. The parameter $\alpha$ is fixed equal to $0.5$, $\theta$ is equal to $0$ on the first row and $10$ on the second row, while $\epsilon$ is respectively equal to $\{0.10, 0.05, 0.01\}$ in the left, center and right columns.}
\label{fig:random-probabilities}
\end{figure}
As for $F(1/3)$ and $\mu = \int x P(\ddr x)$, in Tables \ref{tab:F13} and \ref{tab:Fmu} we report the Kolmogorov distance between $P$ and $P_\epsilon$ for the two sampling algorithms, together with expected value, median, first and third quartiles. For each case and each parameter configuration, we have sampled $10^4$ trajectories from the $\epsilon$-PY process and $10^4$ trajectories from the Pitman--Yor process in the case of $\mu = \int x P(\ddr x)$. As expected, the Kolmogorov distances are generally larger, still close, when using {\sc Algorithm 2} versus {\sc Algorithm 1} due to the approximate nature of the former.

\begin{table}[!ht]
\def~{\hphantom{0}}
\caption{Simulation study on $F(1/3)$}
\label{tab:F13}
\setlength{\tabcolsep}{2pt}
\vspace*{4pt}
\begin{center}
\begin{tabular}{cc@{\hskip 8pt}rr@{\hskip 8pt}rrr@{\hskip 8pt}rrr@{\hskip 8pt}rrr@{\hskip 8pt}rrr}
\hline
  & & \multicolumn{2}{c}{\boldmath${d_K}$} & \multicolumn{3}{c}{\textbf{Mean}} & \multicolumn{3}{c}{\boldmath$25\%$} & 
\multicolumn{3}{c}{\textbf{Median}} & \multicolumn{3}{c}{\boldmath$75\%$} \\ 
$\theta$ & $\epsilon$ & 
{\sc Al1} & {\sc Al2} & 
{\sc Al1} & {\sc Al2} & {\sc PY} & 
{\sc Al1} & {\sc Al2} & {\sc PY} & 
{\sc Al1} & {\sc Al2} & {\sc PY} & 
{\sc Al1} & {\sc Al2} & {\sc PY} \\
\hline
0 & 0.10 & 16.29 & 16.48 & 0.33 & 0.33 & 0.33 & 0.04 & 0.01 & 0.04 & 0.20 & 0.16 & 0.20 & 0.60 & 0.64 & 0.59 \\ 
0 & 0.05 & 11.53 & 12.52 & 0.33 & 0.33 & 0.33 & 0.05 & 0.01 & 0.04 & 0.20 & 0.17 & 0.20 & 0.58 & 0.63 & 0.59 \\ 
  0 & 0.01 & 5.49 & 5.60 & 0.34 & 0.33 & 0.33 & 0.04 & 0.03 & 0.04 & 0.21 & 0.19 & 0.20 & 0.59 & 0.61 & 0.59 \\ 
  1 & 0.10 & 3.08 & 5.65 & 0.33 & 0.33 & 0.33 & 0.14 & 0.12 & 0.14 & 0.29 & 0.28 & 0.28 & 0.49 & 0.50 & 0.49 \\ 
  1 & 0.05 & 1.34 & 3.11 & 0.33 & 0.33 & 0.33 & 0.14 & 0.13 & 0.14 & 0.28 & 0.28 & 0.28 & 0.48 & 0.50 & 0.49 \\ 
  1 & 0.01 & 0.56 & 0.89 & 0.33 & 0.34 & 0.33 & 0.14 & 0.14 & 0.14 & 0.28 & 0.29 & 0.28 & 0.49 & 0.49 & 0.49 \\ 
  10 & 0.10 & 3.10 & 3.81 & 0.33 & 0.33 & 0.33 & 0.25 & 0.25 & 0.26 & 0.32 & 0.32 & 0.32 & 0.40 & 0.41 & 0.40 \\ 
  10 & 0.05 & 1.41 & 1.38 & 0.33 & 0.33 & 0.33 & 0.26 & 0.26 & 0.26 & 0.32 & 0.32 & 0.32 & 0.40 & 0.40 & 0.40 \\ 
  10 & 0.01 & 0.75 & 0.65 & 0.33 & 0.33 & 0.33 & 0.26 & 0.26 & 0.26 & 0.33 & 0.32 & 0.32 & 0.40 & 0.40 & 0.40 \\ 
\hline
\end{tabular}
\end{center}
\par
\medskip
\caption{Simulation study on $\mu = \int x P(\ddr x)$}
\label{tab:Fmu}
\setlength{\tabcolsep}{2pt}
\begin{center}
\begin{tabular}{cc@{\hskip 8pt}rr@{\hskip 8pt}rrr@{\hskip 8pt}rrr@{\hskip 8pt}rrr@{\hskip 8pt}rrr}
\hline
  & & \multicolumn{2}{c}{\boldmath${d_K}$} & \multicolumn{3}{c}{\textbf{Mean}} & \multicolumn{3}{c}{\boldmath$25\%$} & 
\multicolumn{3}{c}{\textbf{Median}} & \multicolumn{3}{c}{\boldmath$75\%$} \\ 
$\theta$ & $\epsilon$ & 
{\sc Al1} & {\sc Al2} & 
{\sc Al1} & {\sc Al2} & {\sc PY} & 
{\sc Al1} & {\sc Al2} & {\sc PY} & 
{\sc Al1} & {\sc Al2} & {\sc PY} & 
{\sc Al1} & {\sc Al2} & {\sc PY} \\
\hline
0 & 0.10 & \hphantom{1}1.60 & \hphantom{1}3.57 & 0.50 & 0.50 & 0.50 & 0.36 & 0.34 & 0.36 & 0.50 & 0.50 & 0.50 & 0.64 & 0.67 & 0.65 \\ 
0 & 0.05 & 0.94 & 2.72 & 0.50 & 0.50 & 0.50 & 0.35 & 0.34 & 0.36 & 0.50 & 0.50 & 0.50 & 0.64 & 0.66 & 0.65 \\ 
  0 & 0.01 & 1.18 & 2.10 & 0.50 & 0.50 & 0.50 & 0.36 & 0.35 & 0.36 & 0.50 & 0.50 & 0.50 & 0.64 & 0.65 & 0.65 \\ 
  1 & 0.10 & 1.61 & 3.18 & 0.50 & 0.50 & 0.50 & 0.40 & 0.39 & 0.40 & 0.50 & 0.50 & 0.50 & 0.60 & 0.61 & 0.60 \\ 
  1 & 0.05 & 1.28 & 2.32 & 0.50 & 0.50 & 0.50 & 0.40 & 0.40 & 0.40 & 0.50 & 0.50 & 0.50 & 0.59 & 0.60 & 0.60 \\ 
  1 & 0.01 & 1.12 & 0.57 & 0.50 & 0.50 & 0.50 & 0.40 & 0.41 & 0.40 & 0.50 & 0.50 & 0.50 & 0.60 & 0.60 & 0.60 \\ 
  10 & 0.10 & 2.81 & 4.18 & 0.50 & 0.50 & 0.50 & 0.45 & 0.46 & 0.46 & 0.50 & 0.50 & 0.50 & 0.55 & 0.55 & 0.54 \\ 
  10 & 0.05 & 1.78 & 1.28 & 0.50 & 0.50 & 0.50 & 0.46 & 0.46 & 0.46 & 0.50 & 0.50 & 0.50 & 0.54 & 0.54 & 0.54 \\ 
  10 & 0.01 & 2.01 & 1.09 & 0.50 & 0.50 & 0.50 & 0.46 & 0.46 & 0.46 & 0.50 & 0.50 & 0.50 & 0.54 & 0.54 & 0.54 \\ 
\hline
\end{tabular}
\end{center}
\captionsetup{labelformat=empty} 
\caption*{Summary statistics for $F(1/3)$ (Table \ref{tab:F13}) and $\mu = \int x P(\ddr x)$ (Table \ref{tab:Fmu}) using {\sc Algorithm 1} ({\sc Al1}) and {\sc Algorithm 2} ({\sc Al2}) to sample from the $\epsilon$-PY process. The Kolmogorov distance ($d_K$) is  between the cumulative distribution functions with respect to the Pitman--Yor (PY) process (multiplied by a factor of 100). The parameter values are $\alpha = 0.5$, $\theta\in\{0,1,10\}$ and $\epsilon\in\{0.10,0.05,0.01\}$.}
\end{table}


\subsection{Computation time}\label{sec:3.3}

In this section, we provide a concrete justification of the computational advantage of using {\sc Algorithm 2} versus {\sc Algorithm 1}.
We simulate $10^4$ $\epsilon$-PY iterations by using {\sc Algorithm 1} and  {\sc Algorithm 2}  for different combinations of the  $\alpha$ and $\theta$ parameters and of the $\epsilon$ error threshold. In Table~\ref{tab:computingtime} (resp. Table~\ref{tab:computingtime_ratio}), we report the average computing time\footnote{The experiments were conducted on an Intel Core i5 processor (3.1 GHz) computer.} \textit{per iteration} (resp. \textit{per support point}) for \textsc{Algorithm 1} and \textsc{Algorithm 2}. By \textit{iteration}, we mean a full realization of the $\epsilon$-PY process including frequencies and locations, while by \textit{support point}, we mean that we divide the total time by the number of support points $\tau(\epsilon)+1$. In order to account for the computational task required per iteration,  the expected stopping time $\E[\tau(\epsilon)]$ is also reported. Both tables illustrate that our proposed approach is faster than  \textsc{Algorithm 1} when the $\epsilon$-PY is composed of about 20 support points or more. The more support points, the faster   \textsc{Algorithm 2} is compared to   \textsc{Algorithm 1}. This disadvantage of the former for small numbers of support points comes from the fixed cost of initially generating a random variable with the same distribution as $T_{\alpha,\theta}$. Conversely, as the number of support points increases, this fixed cost is largely counterbalanced by the fast vector-sampling of a prescribed size, which is in contrast with  \textsc{Algorithm 1}  \texttt{while} loop whose cost increases with the number of support points. This can be seen in Table~\ref{tab:computingtime_ratio} where the actual sampling time per support point is essentially increasing for  \textsc{Algorithm 1} and decreasing for  \textsc{Algorithm 2}. With the parameter configurations tested, \textsc{Algorithm 2}   can be up to 90 times faster \textsc{Algorithm 1} for $\alpha=0.6$, $\theta=10$ and $\epsilon=0.01$.

\begin{table}[!ht]
\def~{\hphantom{0}}
\caption{Computing time (\textit{m}s) per \textbf{iteration}}
\label{tab:computingtime}
\setlength{\tabcolsep}{2pt}
\vspace*{4pt}
\begin{center}
\begin{tabular}{rrrrrrrrrrr}
\cline{3-11}
&&
\multicolumn{3}{c}{\boldmath$\alpha = 0.4$} & 
\multicolumn{3}{c}{\boldmath$\alpha = 0.5$} & 
\multicolumn{3}{c}{\boldmath$\alpha = 0.6$}\\
  $\theta$ & $\epsilon$ & 
  \textsc{Al1} & \textsc{Al2} & $n$ & 
  \textsc{Al1} & \textsc{Al2} & $n$ & 
  \textsc{Al1} & \textsc{Al2} & $n$ \\
 \hline
0 & 0.10 & 0.01 & 0.20 & 5 & 0.02 & 0.04 & 11 & 0.11 & 0.05 & 38 \\ 
  0 & 0.05 & 0.01 & 0.04 & 8 & 0.04 & 0.04 & 21 & 0.20 & 0.06 & 105 \\ 
  0 & 0.01 & 0.04 & 0.04 & 20 & 0.36 & 0.05 & 101 & 15.10 & 0.26 & 1163 \\ 
  1 & 0.10 & 0.03 & 0.19 & 17 & 0.07 & 0.05 & 31 & 0.23 & 0.07 & 92 \\ 
  1 & 0.05 & 0.06 & 0.06 & 26 & 0.13 & 0.06 & 61 & 0.80 & 0.12 & 258 \\ 
  1 & 0.01 & 0.18 & 0.09 & 73 & 0.80 & 0.12 & 301 & 27.75 & 0.57 & 2877 \\ 
  10 & 0.10 & 0.22 & 0.15 & 121 & 0.61 & 0.10 & 211 & 2.11 & 0.18 & 567 \\ 
  10 & 0.05 & 0.45 & 0.10 & 191 & 1.52 & 0.15 & 421 & 9.22 & 0.37 & 1603 \\ 
  10 & 0.01 & 1.93 & 0.20 & 558 & 13.24 & 0.48 & 2101 & 760.68 & 4.01 & 17911 \\ 
     \hline
\end{tabular}
\end{center}
\par
\medskip
\caption{Computing time ($\mu$s) per \textbf{support point}}
\label{tab:computingtime_ratio}
\setlength{\tabcolsep}{2pt}
\begin{center}
\begin{tabular}{rrrrrrrrrrr}
\cline{3-11}
&&
\multicolumn{3}{c}{\boldmath$\alpha = 0.4$} & 
\multicolumn{3}{c}{\boldmath$\alpha = 0.5$} & 
\multicolumn{3}{c}{\boldmath$\alpha = 0.6$}\\
  $\theta$ & $\epsilon$ & 
  \textsc{Al1} & \textsc{Al2} & $n$ & 
  \textsc{Al1} & \textsc{Al2} & $n$ & 
  \textsc{Al1} & \textsc{Al2} & $n$ \\
 \hline
0 & 0.10 & 1.92 & 38.46 & 5 & 1.82 & 3.64 & 11 & 2.91 & 1.32 & 38 \\ 
  0 & 0.05 & 1.30 & 5.22 & 8 & 1.90 & 1.90 & 21 & 1.91 & 0.57 & 105 \\ 
  0 & 0.01 & 1.95 & 1.95 & 20 & 3.56 & 0.50 & 101 & 12.98 & 0.22 & 1163 \\ 
  1 & 0.10 & 1.81 & 11.45 & 17 & 2.26 & 1.61 & 31 & 2.50 & 0.76 & 92 \\ 
  1 & 0.05 & 2.33 & 2.33 & 26 & 2.13 & 0.98 & 61 & 3.10 & 0.46 & 258 \\ 
  1 & 0.01 & 2.45 & 1.23 & 73 & 2.66 & 0.40 & 301 & 9.65 & 0.20 & 2877 \\ 
  10 & 0.10 & 1.82 & 1.24 & 121 & 2.89 & 0.47 & 211 & 3.72 & 0.32 & 567 \\ 
  10 & 0.05 & 2.35 & 0.52 & 191 & 3.61 & 0.36 & 421 & 5.75 & 0.23 & 1603 \\ 
  10 & 0.01 & 3.46 & 0.36 & 558 & 6.30 & 0.23 & 2101 & 42.47 & 0.22 & 17911 \\ 
     \hline
\end{tabular}
\end{center}
\captionsetup{labelformat=empty} 
\caption*{Average computing time per {iteration} (in \textit{milli}second in Table~\ref{tab:computingtime}) and per {support point} (in \textit{micro}second in Table~\ref{tab:computingtime_ratio}) for \textsc{Algorithm 1} (\textsc{Al1}) and \textsc{Algorithm 2} (\textsc{Al2})  based on $10^4$ iterations, and expected stopping time $n=\E[\tau(\epsilon)]$. The parameter values are $\alpha \in\{0.4, 0.5,0.6\}$, $\theta\in\{0,1,10\}$ and $\epsilon\in\{0.10,0.05,0.01\}$.} 
\end{table}


\section{Connections with random partition structures}\label{sec:4}


\subsection{$\alpha$-diversity and asymptotic distribution of $R_n$}\label{sec:4.1}
The random variable $T_{\alpha,\theta}$ in Theorem \ref{th:1} plays a key role in the Pitman--Yor process, in particular for its link with the {\it $\alpha$-diversity} of the process. The $\alpha$-diversity is defined as the almost sure limit of $n^{-\alpha}K_n$ where $K_n$ denotes the (random) number of 
unique values in the first $n$ terms of an exchangeable sequence from $P$ in \eqref{eq:discrete}. According to Theorem 3.8 in \citet{Pit:06}, $n^{-\alpha}K_n\sim_{a.s.}(T_{\alpha,\theta})^{-\alpha}$, in particular, for $\theta=0$, $T_\alpha^{-\alpha}$ has a Mittag-Leffler distribution with $p$-th moment $\Gamma(p+1)/\Gamma(p\alpha+1)$, $p>-1$. According to \citet[Lemma 3.11, eqn (3.36)]{Pit:06}, the asymptotic distribution of the truncation error $R_n$ can be derived from that of $K_n$ to get
  $R_n\sim_{a.s.}
  \alpha (T_{\alpha,\theta})^{-1}
  \,n^{1-1/\alpha}$
as $n\to\infty$.
The proof relies on Kingman's representation of random partitions \citep{Kin:78} together with techniques set forth by \citet{Gne:Han:Pit:07}. In the proof of Theorem \ref{th:1} the asymptotic distribution of $T_n=-\log R_n$ is a direct consequence of the above by an application of the continuous mapping theorem.

When $\theta=0$ it is possible to give an interpretation of the asymptotic distribution of $R_n$ in terms of the jumps of a stable subordinator. In this case the weights of $P$ can be represented as the renormalized jumps of a stable subordinator, with $T_\alpha$ denoting  the total mass. Denote the (unormalized) jumps as $(J_i)_{i\geq 1}$ in decreasing order and as $(\tilde J_i)_{i\geq 1}$ when in size-biased order, 
  $$T_\alpha=\sum_{i\geq 1} J_i
  =\sum_{i\geq 1} \tilde J_i,\quad \text{and }
  T_\alpha R_n=\sum_{i>n} \tilde J_i.$$
By the asymptotic distribution of $R_n$,  
  $n^{1/\alpha-1}
  \sum_{i>n} \tilde J_i\to_{a.s.} \alpha$ 
as $n\to\infty$.
That is, once properly scaled, the small jumps of the stable subordinator (in size-biased random order), interpreted as the \comillas{dust}, converge to the \comillas{proportion} $\alpha$. This is reminiscent to the number of singletons which is asymptotically $(n\to\infty)$ a $\alpha$ proportion of the number of groups in a sample of size $n$, see Lemma 3.11, eqn (3.39), of \citet{Pit:06}.


\subsection{Regenerative random compositions and Anscombe's theorem}\label{sec:4.2}
We review next the connections of the counting renewal process $N(t)$ defined in \eqref{eq:T_n}--\eqref{eq:N(t)} and the theory of regenerative random compositions. The reader is referred to the survey of \citet{Gne:10} for a review. Recall that, when $\alpha=0$ (Dirichlet process case), $V_i\overset{\mbox{\footnotesize \gls{iid}}}{\sim}\text{beta}(1,\theta)$ in the stick-breaking representation \eqref{eq:stick-breaking}, and in turns $Y_i=-\log(1-V_i)\overset{\mbox{\footnotesize \gls{iid}}}{\sim}\mbox{Exp}(\theta)$ and $T_n=-\log R_n\sim\mbox{Gamma}(n,\theta)$. By direct calculus, $N(t)\sim\mbox{Pois}(\theta t)$ so that $\tau(\epsilon)-1=N(\log 1/\epsilon)$ has $\mbox{Pois}(\theta\log 1/\epsilon)$ distribution. 
More generally, the stick-breaking frequencies $(\p_i)_{i\geq 1}$ correspond to the gaps in $[0,1]$ identified by the {\it multiplicative regenerative} set ${\cal R}\subset (0,1)$ consisting of the random partial sums $1-R_k=\sum_{i\leq k}p_i$. The complement open set ${\cal R}^c=(0,1)/\cal R$ can be represented as a disjoint union of countably many open intervals or gaps,
  ${\cal R}^c=\bigcup_{k=0}^\infty(1-R_{k},1-R_{k+1})$, $R_0=1$.
A random composition of the integer $n$ into an ordered sequence $\varkappa_n=(n_1,n_2,\ldots,n_k)$ of positive integers with $\sum_{j}n_j=n$ can be generated as follows: independently of ${\cal R}$, sample $U_1,U_2,\ldots$ from the uniform distribution on $[0,1]$ and group them in clusters by the rule: $U_i,U_j$ belong to the same cluster if they hit the same interval. The random composition of $\varkappa_n$ corresponds then to the record of positive counts in the left-to-right order of the intervals. The composition structure $(\varkappa_n)$ is called regenerative since for all $n>m\geq 1$, conditionally given the first part of $\varkappa_n$ is $m$, if the part is deleted then the remaining composition of $n-m$ is distributed like $\varkappa_{n-m}$. The regenerative set ${\cal R}$ corresponds to the closed range of the multiplicative subordinator $\{1-\exp(-S_t),t\geq0\}$, where $S_t$ is the compound Poisson process with L\'evy intensity $\tilde\nu(\ddr y)=\theta\edr^{-\theta y}\ddr y$. 
Since the range of $S_t$ is a homogeneous Poisson point process on $\R_+$ with rate $\theta$, ${\cal R}$ is an inhomogeneous Poisson point process ${\cal N}(\ddr x)$ on $[0,1]$ with L\'evy intensity
  $\nu(\ddr x)=\theta/(1-x)\ddr x$
so that, for $t=\log 1/\epsilon$,
  $$N(\log 1/\epsilon)={\cal N} [0,1-\epsilon]
  \sim\mbox{Pois}(\lambda),\quad
  \lambda=\int_0^{1-\epsilon}
  \frac{\theta}{1-x}\ddr x=\theta\log 1/\epsilon$$
as expected. Suppose now that $(V_i)_{i\geq 1}$ are independent copies of some random variable $V$ on $[0,1]$, not necessarily $\text{beta}(1,\theta)$ distributed. The corresponding random composition structure has been studied in \citet{Gne:04,Gne:etal:09} as the outcome of a {\it Bernoulli sieve} procedure. We recall here the relevant asymptotic analysis. Let $\mu=\E(-\log(1-V))$ and $\sigma^2=\Var(-\log(1-V))$, equal respectively to $1/\theta$ and $1/\theta^2$ in the DP case, respectively. If those moments are finite, by the \gls{CLT},
  $$\frac{T_n-n\mu}{\sqrt n\sigma}\to_d Z,\quad
  \mbox{as }n\to\infty,$$
where $Z\sim\text{N}(0,1)$, and, by means of Anscobe's Theorem, one obtains that
  $$\frac{N(t)-t/\mu}{\sqrt{\sigma^2t/\mu^3}}
  \rightarrow_d Z,\quad \mbox{as }t\to\infty.$$
It turns out that the normal limit of $N(\log n)$ corresponds to the normal limit of $K_n$,
  $$\frac{K_n-\log n/\mu}{\sqrt{\sigma^2\log n/\mu^{3}}} 
  \rightarrow_d Z,\quad \mbox{as }n\to\infty$$
provided that $\E(-\log V)<\infty$. To see why, consider 
iid random variables $X_1,X_2,\ldots$ with values in $\N$ such that $\{X_i=k\}=\{U_i\in(1-R_{k-1},1-R_{k})\}$. Hence $\Prob(X_1=k|{\cal R})=\p_k$. We then have that $K_n=\#\{k: X_i=k\mbox{ for at least one $i$ among }1,\ldots,n\}$. Define $M_n=\max\{X_1,\ldots,X_n\}$.
For $U_{1,n}\leq U_{2,n}\leq\ldots \leq U_{n,n}$ denoting the order statistics corresponding to the uniform variates $U_1,\ldots,U_n$, we have
  $M_n
  =\min \{j:\ 1-R_j\geq U_{n,n}\}
  =\min \{j:\ T_j\geq E_{n,n}\}$
upon transformation $x\to -\log(1-x)$, where $E_{n,n}$ is the maximum of an iid sample of size $n$ from the standard exponential distribution. Since 
  $N(t)=\max \{n:\ T_n\leq t\}
  =\min \{n:\ T_n\geq t\}-1$ 
we have 
  $M_n-1=N(E_{n,n})$. 
\citet{Gne:etal:09} proves the equivalence
  $$\frac{M_n-b_n}{a_n}\to_d X
  \quad\Longleftrightarrow
  \frac{N(\log n)-b_n}{a_n}\to_d X$$
where $X$ is a random variable with a proper and non degenerate distribution with $a_n>0$, $a_n\to\infty$ and $b_n\in\R$. A key fact exploited in the proof is that, from extreme-value theory, $E_{n,n}-\log n$ has an asymptotic distribution of Gumbel type. That $M_n$ can be replaced by $K_n$ in the equivalence relation above follows from the fact that $M_n-K_n$, the number of integers $k<M_n$ not appearing in the sample $X_1,\ldots,X_n$, is bounded in probability when $\E(-\log V)<\infty$, see Proposition 5.1 in \citet{Gne:etal:09}. 

Back to the Pitman--Yor process case, by Theorem \ref{th:1} we have
  $n^{-\alpha/(1-\alpha)}
  N(\log n)\to_{d}
  (T_{\alpha,\theta}/\alpha)^{-\alpha/(1-\alpha)}$
while
  $n^{-\alpha} K_n
  \to_{a.s.}
  (T_{\alpha,\theta})^{-\alpha}$.
So we see that $N(\log n)$ and $K_n$ do not have the same asymptotic behavior as in the $\alpha=0$ case. 
By using the fact that
  $$\Prob(X_1>n|(\p_i))=R_n,\quad
  R_n\sim_{a.s.}\alpha n^{-(1-\alpha)/\alpha}
  T_{\alpha,\theta}^{-1},$$
and the fact that, conditional on $(\p_i)_{i\geq 1}$, $M_n$ belongs to the domain of attraction of Fr\'echet distribution, \citet[Theorem 6.1]{Pit:Yak:17} establishes that
  $$\Prob(M_n\leq x n^{\alpha/(1-\alpha)})
  \to
  \E\big[\exp\big(-\alpha 
  T_{\alpha,\theta}^{-1}
  x^{-(1-\alpha)/\alpha}\big)\big]$$
so we see that $N(\log n)$ and $M_n$ do not have the same asymptotic behavior as in the $\alpha=0$ case, although they share the same growth rate $n^{\alpha/(1-\alpha)}$. Finally, the non correspondence of the asymptotic distribution of $M_n$ and $K_n$ suggests that the behavior of $M_n-K_n$ is radically different with respect to the $\alpha=0$ case.

\section{Discussion}\label{sec:5}

In this paper we have studied stochastic approximations of the Pitman--Yor process consisting in the truncation of the sequence of stick-breaking frequencies at a random stopping time $\tau(\epsilon)$ that controls the accuracy of the approximation in the total variation distance by $\epsilon$. We name this finite dimensional approximation the $\epsilon$-Pitman--Yor process. We have derived the asymptotic distribution of $\tau(\epsilon)$ as $\epsilon$ goes to zero and we have advanced its use to devise a sampling scheme that generates the stopping time first, and then the frequencies up to that point. The simulations in Section \ref{sec:3} show that the proposed sampler proves computationally very efficient in the moderate to large stopping time regime (for approximately $\tau(\epsilon)\geq 20$).
The asymptotic distribution illustrates how large the stopping time is as the approximation error gets small in terms of the prior parameters $\theta$ and $\alpha$. In particular, it shows that the distribution of $\tau(\epsilon)$ in the Dirichlet process case is not recovered in the limit $\alpha\to 0$ in Theorem \ref{cor:limiting_tau_eps}. In fact, in the Dirichlet process case $\tau(\epsilon)$ grows at a logarithmic rate in $1/\epsilon$ while in Pitman--Yor case it grows at the polynomial rate $\epsilon^{\alpha/(1-\alpha)}$ and the first regime is not recovered by letting $\alpha$ approach $0$ in the second regime. We have also drawn important connections with the theory of random partition structures developed by Jim Pitman and coauthors which highlight the relationship of the the stopping time $\tau(\epsilon)$ with the number $K_n$ of unique values in a sample of size $n$ from the Pitman--Yor process. 

We have left as open problem for future research the study of the conditional distribution of the stick-breaking frequencies given the stopping time. In the Dirichlet process case one can exploit the renewal process interpretation to generate exactly from this conditional distribution. In fact, when $\alpha=0$, the sequence $(-\log R_n)_{n\geq 1}$ corresponds to the jump times of a Poisson process and the conditional distribution of the jumps given the number of jumps at time $t$ can be described in terms of the ordered statistics of i.i.d. uniform random variates on $(0,t)$. The case $\alpha>0$ does not seem to be easily tractable, as it would be if the counting process associated to $\tau(\epsilon)$ were a mixed sample process or, equivalently, a Cox process, cf. \citet[Section 6.3]{Gra:97}.

It would be also interesting to compare the accuracy of our finite dimensional approximation of the Pitman--Yor process to the one proposed in \cite{All:Zar:14}. The latter is based on a representation of the frequencies in decreasing order, cf. \citet[Proposition 22]{Pit:Yor:97}. \cite{All:Zar:14} compare the accuracy of their approximation scheme to a stick-breaking truncation at a number $n$ of stick-breaking frequencies that matches the number of frequencies used in their scheme. Not surprisingly, their approximation is superior since it generates weights in decreasing order, specially when $\alpha$ is large. In contrast, Theorem \ref{cor:limiting_tau_eps} describes precisely how large the truncation threshold $n$ should be as $\alpha$ gets large for a given approximation level $\epsilon$, cf. the center panel of Figure \ref{fig:asymp-approx}. It also underlines that the approximation deteriorates for fixed $n$ and increasing $\alpha$,  which is coherent with the findings in \cite{All:Zar:14}. A fair comparison with their approach can only be done for a given nominal approximation error, but unfortunately the authors did not provide a precise assessment of it. 
The number of stick-breaking frequencies needed to match the approximation accuracy of \cite{All:Zar:14} would be 
{\it de facto} larger due to the non monotonicity. However, since the stopping rule \eqref{eq:tau_eps} adapts to the size of $\alpha$, we do not expect the accuracy of our approximation scheme to deteriorate for $\alpha$ large. 
As for computation time, the techniques used by \cite{All:Zar:14} in order to obtain decreasing frequencies are computational heavy. Their average computing time for $\alpha=0.5$ is about $2.30$ seconds/iteration with $10^4$ locations. This amounts to $0.23$ \textit{milli}seconds/support point, which is $1000$ times slower than the computing time for our Algorithm 2 in the parameter configuration $\alpha=0.5$, $\theta=10$ and $\epsilon=0.01$, equal to $0.23$ \textit{micro}second/support point. It would be interesting to investigate what are the consequences in terms of computation time per iteration for a given approximation error.

\appendix
\section{Random generation of $T_{\alpha,\theta}$}\label{appendix:devroye}

Let $Y$ be a standard $\mbox{Exp}(1)$ random variable. Note that
\begin{align*}
  \Prob((Y/T_\alpha)^\alpha>x)
  &=\int_0^\infty \Prob(Y>x^{1/\alpha} t)f_\alpha(t)\ddr t
  =\int_0^\infty \exp[-x^{1/\alpha}t]f_\alpha(t)\ddr t\\
  &=\E\big[\edr^{-x^{1/\alpha}T_\alpha}\big]
  =\edr^{-x}=\Prob(Y>x)
\end{align*}
so we have
  $Y=_d(Y/T_\alpha)^\alpha$.
For $r<\alpha$, $\E(Y^{-r/\alpha})<\infty$, so we find that
  $\E(Y^{-r/\alpha})=\E( T_\alpha^r)\E(Y^{-r})$
and
\begin{equation}\label{eq:A1}
  \E( T_\alpha^r)=\frac{\E(Y^{-r/\alpha})}{\E(Y^{-r})}
  =\frac{\Gamma(1-r/\alpha)}{\Gamma(1-r)}.
\end{equation}  
The normalizing constant in $f_{\alpha,\theta}(t)$ is
  $\int_0^\infty t^{-\theta}f_\alpha(t)\ddr t
  =\E(T_\alpha^{-\theta})$,
so set $r=-\theta$ and note that $-\theta<\alpha$. 
Let $G_a$ be a gamma random variable with shape $a>0$ and unit rate. Simple moment comparisons using \eqref{eq:A1} yield the distributional equality
  $G_{1+\theta/\alpha}\overset{d}{=}
  (G_{1+\theta}/T_{\alpha,\theta})^\alpha$,
which, however, does not provide a way to generate from $T_{\alpha,\theta}$. For this we resort to \citet{Dev:09}. First we recall how to generate a Zolotarev random variable $Z_{\alpha,b}$ for $\alpha\in(0,1)$ and $b=\theta/\alpha>-1$. Let
  $$C=\frac{\Gamma(1+b\alpha)\Gamma(1+b(1-\alpha))}
  {\pi\Gamma(1+b)}$$
and 
  $$B(u)=A(u)^{-(1-\alpha)}
  =\frac{\sin(u)}
  {\sin(\alpha u)^\alpha\,
  \sin((1-\alpha)u)^{1-\alpha}}.$$
A simple asymptotic argument yields the value
  $B(0)=\alpha^{-\alpha}(1-\alpha)^{-(1-\alpha)}$.
Then
  $f(x)=C\, B(x)^{b},\quad 0\leq x\leq \pi$.
The following bound holds
  $$f(x)\leq 
  C B(0)^b\edr^{-\frac{x^2}{2\sigma^2}},\quad
  \text{with }\sigma^2=\frac{1}{b\alpha(1-\alpha)}.$$
This Gaussian upper bound suggests a simple rejection sampler for sampling Zolotarev random variates. Following \citet{Dev:09}, it is most efficient to adapt the sampler to the value of $\sigma$. If $\sigma\geq\sqrt{2\pi}$, rejection from a uniform random variate is best. Otherwise, use a normal dominating curve as suggested in the bound above. The details are given below.

\begin{quote}
\centerline{{\sc Algorithm 3} (Sampler of $T_{\alpha,\theta}$)}
\vspace{4pt}

{\tt
1. set $b=\theta/\alpha$ and $\sigma=\sqrt{b\alpha(1-\alpha)}$
\vspace{4pt}

2. if $\sigma\geq\sqrt{2\pi}$:
\\
\phantom{1. }then repeat:$\quad$ generate $U\sim\text{\rm Unif}(0,\pi)$ and $V\sim\text{\rm Unif}(0,1)$. 
\\
\phantom{1. then repeat:$\quad$} set $X\leftarrow U$, $W\leftarrow B(X)$.
\\
\phantom{1. then} until $V\leq (W/B(0))^b$
\\
\phantom{1. }else repeat:$\quad$ generate $N\sim\text{\rm N}(0,1)$ and $V\sim\text{\rm Unif}(,1)$. 
\\
\phantom{1. else repeat:$\quad$} set $X\leftarrow \sigma|N|$, $W\leftarrow B(X)$.
\\
\phantom{1. else} until $X\leq \pi$ and $V\edr^{-N^2/2}\leq (W/B(0))^b$
\vspace{4pt}

3. generate $G\overset{d}{=}
G_{1+b(1-\alpha)/\alpha}$
\vspace{4pt}

4. set $T\leftarrow 1/(WG^{1-\alpha})^{1/\alpha}$
\vspace{4pt}

5. return $T$
}
\end{quote}

\bibliographystyle{apalike}
\bibliography{biblio}

\end{document}